\documentclass[a4paper,10pt]{article}

\usepackage{amsfonts,amssymb,amsmath}
\usepackage{dsfont}
\usepackage{cancel}
\usepackage{array}
\usepackage{mathtools}

\input{xy}
\usepackage[all]{xy}
\xyoption{all}
\xyoption{poly}

\newtheorem{thm}{Theorem}[section]

\newtheorem{lem}[thm]{Lemma}
\newtheorem{pro}[thm]{Proposition}
\newtheorem{defi}[thm]{Definition}
\newtheorem{rem}[thm]{Remark}
\newtheorem{exa}[thm]{Example}

\newenvironment{proof}{\noindent \textbf{{Proof.}} \sf}

\def\qed{\hfill $\diamond$ \bigskip}

\newcommand{\li}[2]{{\vphantom{#2}}_{#1}#2}

\def\B{{\mathcal B}}
\def\C{{\mathcal C}}

\def\F{{\mathcal F}}
\def\H{{\mathcal H}}

\def\L{{\mathcal L}}

\def\T{{\mathcal T}}

\def\lim{\mathop{\rm lim}\nolimits}

\def\deg{\mathsf{deg}}

\begin{document}

\sf

\title{The fundamental group of a Hopf linear category\\ \vskip7mm \footnotesize Dedicated to Eduardo N. Marcos for his 60th birthday}
\author{Claude Cibils and Andrea Solotar
\thanks{\footnotesize This work has been supported by the projects  UBACYT 20020130100533BA, PIP-CONICET 112--201101--00617, PICT 2011--1510 and MATHAMSUD-GR2HOPF.
The second author is a  research member of
CONICET (Argentina).}}

\date{}

\maketitle
\begin{abstract}
 We define the fundamental group of a Hopf algebra over a field. For this purpose we first consider gradings of Hopf algebras and Galois coverings. The latter are given by linear categories with new additional structure which we call Hopf linear categories over a finite group. We compare this invariant to the fundamental group of the underlying linear category, and we compute those groups for  families of examples.
\end{abstract}


\section{\sf Introduction}
The main purpose of this paper is to initiate the theory of the fundamental group of a Hopf algebra over a field, relative to a finite group. In this context there is  no analogous of homotopy theory of loops available, so the situation differs from the algebraic topology setting. Instead we use the approach developed in recent years for considering an intrinsic fundamental group "\`{a} la Grothendieck" of an associative algebra over a field or more generally of a small category over a field, see  \cite{CRS JNG, CRS ANT 10, CRS ART 12}. As mentioned in \cite{CRS JNG}, we  follow methods closely related to the way in which the fundamental group is considered in algebraic geometry, after A. Grothendieck and C. Chevalley, see for instance \cite{dodo}.

Gradings by different groups of a given Hopf algebra constitute the main tool, they approximate the fundamental group that we define in this paper. More precisely, in order to retain only useful information, we  consider connected gradings, namely gradings where the degrees of the homogeneous components generate the grading group. They correspond to connected Galois coverings, see below.

 A grading of a Hopf algebra $H$ is a standard notion, see for instance  \cite{mon}.  The specific point is that the group which grades is required to be abelian in order to insure that it also grades $H\otimes H$ so that the comultiplication $\Delta : H\longrightarrow H\otimes H$ is homogeneous. This  is a main difference from the associative algebra context, where the groups grading an algebra are arbitrary. In turn, this implies that the fundamental group that we introduce is abelian.

In case the abelian group $\Gamma$ grading a Hopf algebra is finite, we provide a Galois covering using a smash product category with Galois group $\Gamma$ -- see also E. Beneish and W. Chin \cite{chin}. The resulting linear category   has an additional structure that we consider in Section \ref{hopf categories}. More precisely, we define a Hopf $k$-category over a finite group $G$ to be a small $k$-category whose objects are the elements of $G$, equipped with a comultiplication, a counit and an antipode. In particular a Hopf $k$-category over a trivial group is a Hopf $k$-algebra. The smash product of a Hopf $k$-category $\H$ over a finite group $G$ with respect to a finite abelian group $\Gamma$ is a Hopf $k$-category over $G\times \Gamma$.

V.G. Turaev in \cite{tura,tura2} (see also A. Virelizier in \cite{vire}) has considered Hopf $G$-coalgebras where $G$ is a discrete group. We will recall the definition in order to  prove the following. Assuming $G$ is finite, the Hopf $G$-coalgebras are precisely Hopf $k$-categories over $G$ with zero space of morphisms between different objects, see Proposition \ref{turaev}.

The fundamental group of $\H$ is obtained by considering all  Hopf connected gradings by finite abelian groups -- which provide Galois coverings -- and morphisms between them. An element of the fundamental group is a family of elements belonging to the groups which are grading $\H$ in a connected way, and which is compatible through morphisms of gradings. Consequently this group is abelian. This property can be compared with the fact that the usual fundamental group of an $H$-space is abelian. Recall that an $H$-space is a topological space with additional structure, namely an associative continuous product with a neutral element, see for instance \cite{hatcher}. Note that a Hopf $k$-category is a $k$-category with additional structure as well, namely a comultiplication, a counit and an antipode.

Gradings of Lie algebras have been considered in recent years, see for instance \cite{baza,eldu,koch,pareza}. In particular gradings on finite dimensional semisimple Lie algebras by cyclic groups are useful for Kac-Moody algebras, see \cite{kac}. It is clear that a grading of a Lie algebra provides a Hopf grading of its enveloping algebra. Hence a theory of the fundamental group of a Lie algebra could be considered, together with a relation to the fundamental group that we introduce in this paper for Hopf algebras.

 We observe that the fundamental group of the trivial Hopf $k$-category over any finite group is trivial,  and that it is an invariant of the isomorphism class of a Hopf $k$-category. Moreover if a universal grading of a Hopf $k$-category exists,  its  grading  group is isomorphic to the fundamental group. In addition we prove the existence of a group homomorphism $\tau$ from the fundamental group of the underlying $k$-category to the fundamental group of the Hopf $k$-category. We show that this morphism is surjective if the Hopf $k$-category admits both a universal and a Hopf universal grading.

 For specific $k$-categories with a connected grading, there exists an injective Hurewicz morphism from the linear abelianisation of the grading group to the first Hochschild cohomology group of the $k$-category, see  \cite{CRS DOC 11}. The theory introduced in this paper could have a relation with the first cohomology of a Hopf algebra as considered in \cite{gesc,ta}.

In the last sections we compute the fundamental group of Taft categories, and of the Hopf algebras $k^{C_n}$ where $C_n$ is a cyclic group, for $n\leq 7$. It is interesting to note that  the fundamental group of a Taft category is a finer invariant than the fundamental group of its underlying $k$-category. Indeed, it takes into account the number of objects while the latter is the infinite cyclic group regardless of the number of objects. For Taft categories $\tau$ is injective.

\section{\sf Hopf linear categories}\label{hopf categories}

Let $k$ be a field.
A \emph{$k$-category} $\B$ is a small category with set of objects $\B_0$ which is enriched over $k$-vector spaces.
In other words, for $x,y \in \B_0$ the set of morphisms ${}_y\B_x$ from $x$ to $y$ is a
$k$-vector space and composition of morphisms is $k$-bilinear.  The endomorphism spaces are $k$-algebras and each morphism space  ${}_y\B_x$ is a  ${}_y\B_y -
{}_x\B_x$-bimodule.

\begin{defi}
Let $\Lambda$ be a $k$-algebra and let $E$ be a finite set of orthogonal idempotents of $\Lambda$ which is complete, i.e. $\sum_{e\in E}e=1$. The \emph{Peirce category} of $\Lambda$ with respect to $E$ is the $k$-category $\B_{\Lambda,E}$ whose set of objects is $E$ and ${}_y(\B_{\Lambda,E})_x= y \Lambda x$ for all $x,y\in E$.
Composition is inferred from the product in $\Lambda$.
\end{defi}

The reverse construction for a $k$-category $\B$ with a finite number of objects is  the \emph{sum-algebra} $\oplus\B$  which is the direct sum of all vector spaces of morphisms,  provided with the product inferred from the matrix product combined with composition in $\B$.  It provides a $k$-algebra equipped with a finite complete set of orthogonal idempotents, given by the identities of the objects.

Observe that this way $k$-categories with finite number of objects are identified with $k$-algebras equipped  with a finite complete set of orthogonal idempotents.

\begin{defi}\label{Hopf k category}
Let $G$ be a finite group. A \emph{Hopf $k$-category} over $G$ is a $k$-category $\H$ whose objects are the elements of $G$, provided with the following additional data for each pair of objects $x,y\in G$.
\begin{enumerate}
\item A $k$-linear comultiplication map
\[ {}_y\Delta_x: {}_y\H_x \longrightarrow \bigoplus_{\substack{ x''x'=x \\ y''y'=y }} \li{y''}{\H}_{x''}\otimes \li{y'}{\H}_{x'}.\]
In the following we make use of the following convention: if two morphisms $g$ and $f$ cannot be composed because the target object $\tau(f)$ of $f$ is different from the source object $\sigma(g)$ of $g$, then $gf=0$.\\
The comultiplication verifies that whenever $g\in {}_z\H_{y}$ and $f \in {}_y\H_{x}$ are composable,
$${}_z\Delta_x(gf)= {}_z\Delta_y(g){}_y\Delta_x(f).$$
 Moreover
\[ {}_x\Delta_x({}_x1_x)= \sum_{\substack{x''x'=x}}\li{x''}{1_{x''}}\otimes {}_{x'}1_{x'}. \]

\item A counit
\[ {}_y\epsilon_x: {}_y\H_x \to k, \]
which is a $k$-functor from $\H$ to the single object category with endomorphism algebra reduced to $k$.
\item An antipode $S$ which is a contravariant functor of $\H$ such that $S(x)=x^{-1}.$
\end{enumerate}
The previous maps verify the conditions below, analogous  to those defining  a Hopf $k$-algebra. We state them using using Sweedler's notation
\[ \Delta\left( {}_yf_x\right)= \sum_{\substack{ x''x'=x\\ y''y'=y }}\li{y''}{f''}_{x''}\otimes \li{y'}{f'}_{x'}.\]
\begin{enumerate}
\item Coassociativity: given objects $x$ and $y$, both maps from ${}_y\H_{x}$ to
\[ \bigoplus_{\substack{x'''x''x'=x\\ y'''y''y'=y} }{}_{y'''}\H_{x'''}\otimes {}_{y''}\H_{x''}\otimes {}_{y'}\H_{x'} \]
induced by the comultiplication coincide.

\item Counit: if $x$ and $y$ are objects such that at least one of them is different from $1_G,$ then ${}_y\epsilon_x=0$, and given  ${}_yf_x \in {}_y\H_x$,
$$ {}_1\epsilon_1(\li{1}{{f''}_1})\ \  \li{y}{{f'}_x}= {}_yf_x \mbox{ and } \li{y}{{f''}_x}\ \ {}_1\epsilon_1(\li{1}{{f'}_1})= {}_yf_x.$$

\item Antipode: for ${}_yf_x \in {}_y\H_x$,
\[ \sum_{\substack{x''x'=x\\ y''y'=y}}S(\li{y''}{f''_{x''}})\ \li{y'}{f'_{x'}}={}_y\epsilon_x({}_yf_x) \mbox{ and } \sum_{\substack{x''x'=x\\ y''y'=y}}\li{y''}{f''_{x''}}\ S(\li{y'}{f'_{x'}})={}_y\epsilon_x({}_yf_x). \]
\end{enumerate}
\end{defi}

\begin{exa}\label{trivial}
The trivial  Hopf $k$-category over a finite group $G$ is denoted $\mathcal{T}_G$, it has $G$ as set of objects, zero morphism space between different objects while  ${\li{x}{\left(\mathcal{T}_G\right)}}_x =k$ for each object $x$. We set
\begin{equation}
{}_x\Delta_x({}_x1_x)= \sum_{\substack{x''x'=x}}\li{x''}{1_{x''}}\otimes {}_{x'}1_{x'},
\end{equation}
while if $x\neq $1 we put $\epsilon\left( {}_x1_x\right)=0$, and $\epsilon\left({}_11_1\right)=1.$ The antipode is given by $S\left({}_x1_x\right)=\li{x^{-1}}{1}_{x^{-1}}.$
\end{exa}

{For a discrete group $G$ we recall now briefly from \cite{vire} the definition of a  \textit{Hopf $G$-coalgebra} $\mathbf H$, that is a family $\mathbf H=\{ \mathbf H_s\}_{s\in G} $ of $k$-algebras endowed with a family of comultiplications $\mathbf\Delta=\{\mathbf\Delta_{s,t}:\mathbf H_{st}\to \mathbf H_s\otimes  \mathbf H_t\}_{s,t \in G},$ a counit $\epsilon :  \mathbf H_1\to k,$ and an antipode $S=\{S_s:\mathbf H_s\to \mathbf H_{s^{-1}}\}_{s\in G}$ verifying properties and compatibility conditions which are adapted from those defining a Hopf $k$-algebra.

\begin{pro}\label{turaev}
Let $G$ be a finite group. There is a one-to-one correspondence between Hopf $G$-coalgebras and Hopf $k$-categories over $G$ which have zero morphisms between different objects.
\end{pro}

\begin{proof}
Let $\mathbf H$ be a Hopf $G$-coalgebra. The associated $k$-category $\H$ has set of objects $G$, the algebra of endomorphisms of an object $s$ is the $k$-algebra $\mathbf H_s$ while the morphisms between different objects are reduced to $0$. The composition in the category is given at each object by the product of the $k$-algebra $\mathbf H_s$. In order to define the comultiplications in $\H$, observe that we just need a map
$${}_s\Delta_s : {}_s\H_s = \mathbf H_s \longrightarrow \bigoplus_{s''s'=s} {}_{s''}\H_{s''}\otimes {}_{s'}\H_{s'}=\bigoplus_{s''s'=s} \mathbf H_{s''}\otimes \mathbf H_{s'}$$
since there are no non zero morphisms between different objects in $\H$.
Consider   $${}_s\Delta_s=\bigoplus_{s''s'=s}\mathbf\Delta_{s'',s'}$$
which is clearly coassociative in the sense of Definition \ref{Hopf k category}. Similarly we define a counit and an antipode out of the data contained in $\mathbf H.$

Reciprocally, consider a Hopf $k$-category $\H$ over a group $G$ which has zero morphisms between different objects. Let $\mathbf H$ be the family of $k$-algebras  $\{\mathbf H_s={}_s\H_s\}_{s\in G}.$ The comultiplication of $\H$ provides maps
\[ {}_s\Delta_s: {}_s\H_s \longrightarrow \bigoplus_{\substack{ s''s'=s \\ s''''s'''=s }} \li{s''''}{\H}_{s''}\otimes \li{s'''}{\H}_{s'}\]
which reduce to
\[ {}_s\Delta_s:\mathbf H_s=  {}_s\H_s \longrightarrow \bigoplus_{\substack{ s''s'=s}} \li{s''}{\H}_{s''}\otimes \li{s'}{\H}_{s'}= \bigoplus_{\substack{ s''s'=s}} {\mathbf H}_{s''}\otimes {\mathbf H}_{s'}.\]
For each pair of elements $(s'', s')$ we project on the corresponding direct summand and the resulting map is settled as
$$\mathbf \Delta_{s'',s'} : \mathbf H_{s''s'} ={}_{s''s'}\H_{s''s'}\longrightarrow  {\mathbf H}_{s''}\otimes {\mathbf H}_{s'}$$
which verifies the coassociativity condition. Similarly, we define the counits and the antipodes obtaining thus a Hopf $G$-coalgebra. Clearly, both constructions are inverse to each other.
\qed
\end{proof}

\begin{rem}
\label{dualofagroup}
Recall that for a finite group $G$, the commutative associative algebra $k^G$ of $k$-valued maps over $G$ is a Hopf algebra which is the dual of the group algebra $kG$. For any $a:G\to k$, we have $\Delta a(s_1,s_2)=a(s_1s_2)$, while $S(a)(s)=a(s^{-1})$ and $\epsilon(a) = a(1).$

The Dirac mass $\delta_s$ on an element $s\in G$ is the zero map on each element of $G$ except on $s$ where its value is $1$. The set of Dirac masses is the complete system of primitive orthogonal central idempotents of $k^G$. In fact $\Delta \delta_s =\sum_{s''s'=s}\delta_{s''}\otimes \delta_{s'}.$ Moreover $S(\delta_s)=\delta_{s^{-1}}$. For $s\neq 1$ we have $\epsilon(\delta_s)=0$ while $\epsilon(\delta_1)=1.$
\end{rem}

\begin{lem}
Let $\H$ be a Hopf $k$-category over a finite group $G$. The sum-algebra $\oplus\H$ is a Hopf algebra which contains canonically the Hopf algebra $k^G$.
\end{lem}

\begin{proof}
The direct sum of the comultiplications of $\H$ on  the vector spaces of morphisms provides an algebra map $$\oplus\H\to\left(\oplus\H\right) \otimes\left(\oplus \H\right)$$ which is coassociative. Similarly the antipodes of $\H$ provides an antipode of $\oplus\H$ and the direct sum of the counits gives a counit $\epsilon :\oplus\H\to k$.

Consider the linear map which sends $\delta_s$ to the element of $\oplus\H$ given by the identity endomorphism of the object $s$ completed by zeros in the other components of the direct sum. This is an injective map of Hopf algebras.\qed
\end{proof}}

\begin{thm}\label{onetoone}
Let $G$ be a finite group and let $k$ be a field. There is a one-to-one correspondence between Hopf $k$-categories over $G$ and the embeddings of $k^G$ in Hopf $k$-algebras.
\end{thm}
\noindent\textbf{Proof. }\sf
Let $\H$ be a Hopf $k$-{category} over $G$ and
let $\oplus \H$ be the corresponding sum-algebra. In the preceding Lemma we have proved that
 $\H$ provides a Hopf algebra structure on $\oplus \H$  with a canonical inclusion of $k^G$ in $\oplus\H$.

Reciprocally, consider an embedding of Hopf $k$-algebras $k^G\subset H$.The set $E=\{\delta_x\}_{x\in G}$ is a complete finite system of orthogonal idempotents of $H$.

Let $\B_{H,E}$ be the corresponding Peirce category, whose objects are in one-to-one correspondence with elements of $G$.
The comultiplication $\Delta$ of $H$ restricted to $\delta_yH\delta_x={\li{y}{\left(\B_{H,E}\right)}_x}$
provides ${}_y\Delta_x$. The family $({}_y\Delta_x)_{x,y\in G}$ verifies the requested properties.
Similarly, the counit $\epsilon:H\to k$ restricts to
${}_y\epsilon_x: \delta_yH\delta_x \to k$ and the antipode $S$ provides
${}_yS_x: \delta_yH\delta_x \to \delta_{x^{-1}}H\delta_{y^{-1}}$.\qed

\begin{rem} Let $G$ be a finite group. In the correspondence of the Theorem above, the identity embedding of $k^G$ into itself corresponds to  the trivial Hopf $k$-category ${\mathcal T}_G$ over $G$, see Example \ref{trivial}.
\end{rem}
\section{\sf Gradings and the smash product}\label{gradings}

We first recall the definition of a connected grading of a $k$-category $\B$ and the corresponding smash product, see for instance \cite{CM}.

\begin{defi} \label{gradings}\sf   A \emph{grading} $X$ of  $\B$ by a group $\Gamma_X$ is a  direct sum decomposition of each morphism space
\[ {}_{y}\B_x = \bigoplus_{s\in \Gamma_X}X^s{}_{y} \B_x\]
compatible with composition: for $x,y,z\in \B_0$ and $s,t\in \Gamma_X$,
\[\left(X^t{}_{z}\B_{y}\right)\left(X^s{}_{y}\B_x\right) \subset X^{ts}{}_{z}\B_x.\]
In case $X^s{}_{y}\B_x$ is non zero, this vector space is called the \em{homogeneous component} of degree $s$ from $x$ to $y$.

The group $\Gamma_X$ is called the \emph{grading group} of $X$.
\end{defi}

A \emph{virtual morphism} is a pair $(f,\epsilon)$ where $f$ is a morphism and $\epsilon = 1\mbox{ or } -1$. Source and target objects of virtual morphisms remain unchanged if $\epsilon=1$ and are reversed if $\epsilon=-1$. A \emph{walk} from $x$ to $y$ is a sequence
$w=(f_n,\epsilon_n) \dots (f_1,\epsilon_1)$  where the morphisms $f_i$ are non zero,
$\epsilon_i=\pm 1$ and
$\tau(f_i,\epsilon_i)=\sigma(f_{i+1},\epsilon_{i+1})$ for $i=1,\dots, n-1$, with  $\sigma(f_1,\epsilon_1)=x$ and $\tau(f_n,\epsilon_n)=y$.

A \emph{$X$-homogeneous walk} $w$ from $x$ to $y$ in $\B$ is a sequence as above, where each
$f_i$ is $X$-homogeneous, of $X$-degree denoted $\mathsf{deg}_Xf_i$.
The $X$-degree of $w$ is defined as follows:
\[ \mathsf{deg}_Xw= (\mathsf{deg}f_n)^{\epsilon_n}\dots (\mathsf{deg}f_1)^{\epsilon_1}. \]
The set of $X$-homogeneous walks from $x$ to $y$ is denoted ${}_yHW_X(\B)_x$.

We assume that $\B$ is connected, that is from any object $x$ in $\B$ we can reach any other object $y$ by a non zero walk.
The grading $X$ of $\B$ is  \emph{connected} if given any two objects $x,y$ in $\B$,
the degree map
\[ \mathsf{deg}_X: {}_y HW_X(\B)_x\to \Gamma_X \]
is surjective.

The following result is straightforward.
\begin{lem}\label{support}
Let $\B$ be a single object $k$-category graded by a group $\Gamma$,  equivalently let $B$ be a graded $k$-algebra. The grading is connected if and only if its support (that is the set of  degrees of the homogeneous components of $B$) generate  $\Gamma$.
\end{lem}

Gradings and Galois coverings are related as follows.
A \emph{smash product category} $\B\#X$ is inferred from a grading $X$ of $\B$, see \cite{CM, CRS ANT 10}:

\[ (\B\# X)_0= \B_0 \times \Gamma_X  \mbox{ and }  {}_{(y,t)}(\B\# X)_{(x,s)}= X^{t^{-1}s}{}_y\B_x. \]

One can easily check that there is a well-defined composition inherited from composition in $\B$, which makes $\B\# X$
a $k$-category. Moreover, if $\B$ and $X$ are connected, the category $\B\#X$ is connected
and the natural functor $F_X:\B\# X\to \B$ is a Galois covering of $\B$, see \cite{CM}.
Up to isomorphism, this construction provides all Galois coverings of $\B$.

In case there exists a \emph{universal} grading, see for instance  \cite{CRS JNG}, its group is the fundamental group of $\B$. This occurs for
some families of categories \cite{CRS ART 12} but in general a universal grading do not exist. In order to deal with the general situation,  in \cite{CRS ART 12, CRS JNG} a fundamental group \emph{\`a la Grothendieck} is defined as follows.

A \emph{morphism} $\mu:X\to Y$ \emph{of connected gradings} is a group morphism
$\mu:\Gamma_X \to \Gamma_Y$ such that there exists a homogeneous  invertible endofunctor $J:\B \to \B$ which is the identity on objects and such that  for a chosen $b_0\in \B_0$, the diagram

\[
\xymatrix{
{}_{b_0}HW_X(\B)_{b_0} \ar@{->>}[d]_{\deg_X} \ar[r]^{HW(J)} & {}_{b_0}HW_Y(\B)_{b_0}\ar@{->>}[d]^{\deg_Y}\\
\Gamma_X \ar[r]_{\mu}   &\Gamma_Y}
\]

commutes. Note that in general $J$ is not uniquely determined by $\mu$.

\begin{rem}\label{surjective}
We observe that if $\mu:X\to Y$ is a morphism of connected gradings then $\mu:\Gamma_X\to \Gamma_Y$ is surjective.
\end{rem}

\begin{defi}\label{fundgroup}
The \emph{fundamental group} $\Pi_1(\B,b_0)$ is the set of families $\{\gamma_X\}$,
where $\gamma_X\in \Gamma_X$ and $X$ varies amongst the connected gradings of $\B$,
such that for any morphism of connected gradings $\mu:X \to Y$ {the} equality
$\mu(\gamma_X)=\gamma_Y$ holds.  The product of these families is the pointwise product.
\end{defi}

\medskip

Next we  recall the definition of a graded Hopf algebra, see S. Montgomery \cite[10.5]{mon}. Note that M. Aguiar and S. Mahajan in \cite{aguiar} consider graded Hopf algebras as a setting where their results can be extended. Graded Hopf algebras are also used by F. Patras in \cite{patras}.

\begin{defi}
Let $H$ be a Hopf $k$-algebra. A \emph{grading} $X$ of $H$ by an abelian group $\Gamma_X$ is a grading of the underlying associative algebra by $\Gamma_X$, which verifies that the comultiplication, the  counit and  the antipode are homogeneous.
\end{defi}
If $k^G\subset H$ is an embedding of Hopf algebras, a grading as above which verifies in addition that $k^G$ is of trivial degree is called a $k^G$-grading.

\begin{rem}
As mentioned in the introduction, the fact that $\Gamma_X$ is abelian ensures that $H\otimes H$ is a graded algebra. Moreover it is well known that in this situation the category of graded $H$-modules is monoidal.
\end{rem}
\normalsize
\begin{defi}
Let $\H$ be a Hopf $k$-category over $G$. A \emph{Hopf grading} $X$ of $\H$ by an abelian group $\Gamma_X$ is a structure obtained by transporting a $k^G$- grading of $\oplus\H$ using Theorem \ref{onetoone}.
\end{defi}

\begin{thm}
Let $\H$ be a Hopf $k$-category over a finite group $G$. Let $X$ be a Hopf grading of $\H$ by an abelian group $\Gamma_X$.
If $\Gamma_X$ is finite, then the smash product category $\H\# X$ is a Hopf $k$-category over the group $G\times \Gamma_X$.
\end{thm}

\begin{rem} We require that the grading group is finite in order to insure that the group $G\times \Gamma_X$ of the  Hopf $k$-category obtained  via the smash product is finite.
\end{rem}
\noindent\textbf{Proof. }\sf
For simplicity we consider a single object category $\H$, hence the set of objects of the smash product is  $\Gamma_X$.

Each homogeneous element $f$ of $\H$ gives rise to a family of morphisms in $\H\#X$
\[\left\{  {}_{t} f_{ t\deg_X(f)}\right\}_{t\in\Gamma_X}.\]
Letting $f$ vary we obtain all the morphisms of $\H\# X$.
On the other hand, the coproduct of $\H$ provides $$\Delta(f)=\sum f''\otimes f',$$ where $f{''}$ and $f'$ are homogeneous and verify $\deg_X(f)=\deg_X(f{''})\deg_X(f')$.

In order to define the coproduct of a morphism, we consider the family
\[\left\{         \li{v}{ f''}_{v\deg_X(f'')}  \otimes   {}_{{}_u}{ f'}_{u\deg_X(f')} \right\}\]
with $u, v \in \Gamma_X$. Our purpose is to set a Hopf category structure on $\H\#X$, consequently we retain the members of this family verifying
 $$vu=t \mbox{ and } v\deg_X(f{''}) u \deg_X(f')= t \deg_X(f).$$ Since $\Gamma_X$ is abelian, the second requirement is a consequence of the first one. The  coproduct of the smash product is thus defined as follows
\[\Delta \left({}_{t} f_{ t\deg_X(f)}\right)= \sum_{t,x} \ \li{tx^{-1}} {f''}_{\ tx^{-1}\deg_X(f'')}  \otimes   \li{x}{f'}_{x\deg_X(f')} .\]

In order to prove that $\Delta$ is coassociative, consider the following computations:

\[(1\otimes \Delta)\Delta \left({}_{t} f_{ t\deg_X(f)}\right) =   \sum_{t,x,y} \  \ \li{tx^{-1}} {f'''}_{\ tx^{-1}\deg_X(f''')}  \otimes\li{xy^{-1}} {f''}_{\ xy^{-1}\deg_X(f'')}  \otimes   \li{y}{f'}_{y\deg_X(f')}, \]

\[\Delta (\Delta\otimes 1) \left({}_{t} f_{ t\deg_X(f)}\right) = \sum _{t,u,v} \  \ \li{tu^{-1}v^{-1}} {f'''}_{tu^{-1}v^{-1}\deg_X(f''')}  \otimes \li{v}{f''}_{v\deg_X(f'')} \otimes   \li{u}{f'}_{u\deg_X(f')} .\]

Replacing $v$ by $xy^{-1}$ and $y$ by $u$ both expressions coincide. The verification for the antipode and the counit are straightforward.
\qed

Recall that the category of  left modules over a Hopf algebra $H$ admits a monoidal structure given by the tensor product over the base field $k$ and the comultiplication $\Delta$ on $H$.
 Considering,  as usual, linear functors from a Hopf $k$-category $\H$ to $k$-vector spaces as left $\oplus\H$-modules, provides an immediate extension of this construction and a monoidal structure on such functors which are simply called $\H$-modules.

Using the above results the proof of the following result is clear.

\begin{thm}
Let $\H$ be a Hopf $k$-category over $G$ with a grading $X$ whose structure group $\Gamma_X$  is finite and abelian. There is an isomorphism of monoidal categories between $\Gamma_X$-graded modules over $\H$, and modules over the smash product $\H\#X.$
\end{thm}

\section{\sf Fundamental group of a Hopf linear category}

In this section we will follow the lines of \cite{CRS JNG} as described in the previous section, in order to define the fundamental group of a Hopf $k$-category $\H$ over a finite group $G$ with identity element $1$. There is a morphism $\tau$ of the fundamental group of $\H$ to this new fundamental group which takes into account the Hopf structure of $\H$. In general $\tau$ is not surjective neither injective but we prove that in case $\H$ admits both universal grading and universal Hopf grading, then it is surjective.
\begin{defi}\label{morphism Hopf gradings}
Let $X$ and $Y$ be Hopf connected gradings of $\H$. A morphism $\mu: X\to Y$ is a morphism of groups $\mu: \Gamma_X \to \Gamma_Y$ such that there exists at least one homogeneous invertible endofunctor $J:\H \to \H$ of Hopf $k$-categories such that the following diagram commutes:

\[
\xymatrix{
{}_{1}HW_X(\B)_{1} \ar@{->>}[d]_{\deg_X} \ar[r]^{HW(J)} & {}_{1}HW_Y(\B)_{1}\ar@{->>}[d]^{\deg_Y}\\
\Gamma_X \ar[r]_{\mu}   &\Gamma_Y}
\]
\end{defi}

\begin{rem}\label{surjectiveHopf}
As observed for gradings of a $k$-category, if $\mu:X\to Y$ is a morphism of Hopf connected gradings then $\mu:\Gamma_X\to \Gamma_Y$ is surjective.
\end{rem}

As mentioned  in the Introduction, following methods closely related to the way in which the fundamental group is considered in algebraic geometry, we define the fundamental group of a Hopf $k$-category $\H$ as follows.

\begin{defi}\label{hopffundamental}
Let $\H$ be a Hopf $k$-category over a finite group $G$. An element of the fundamental group $\Pi_1^ {\mathsf H}(\H)$ is a family $\{\gamma_X\}$ where $X$ varies amongst the connected Hopf gradings of $\H$ with abelian finite grading group,
$\gamma_X\in \Gamma_X$, and the family verifies $\mu(\gamma_X)= \gamma_Y$ for every morphism of Hopf gradings $\mu:X\to Y$ .
The product of $\Pi_1^ {\mathsf H}(\H)$ is the pointwise product of the families.
\end{defi}

\begin{rem}
The fundamental group of a Hopf $k$-category over a finite group $G$ is abelian since all the $\Gamma_X$'s are abelian.
\end{rem}

\begin{lem}
The fundamental group of the trivial Hopf $k$-category over a finite group is zero.
\end{lem}

Observe that if  $F:\H\to\H'$ is an isomorphism of Hopf $k$-categories over the same group $G$, then there is an isomorphism  $\Pi_1^ {\mathsf H}(\H)\to \Pi_1^ {\mathsf H}(\H').$

\begin{thm}\label{themorphism}
Let $\H$ be a Hopf category over a finite group $G$. There is a group homomorphism
$$\tau:\Pi_1(\H, 1)\longrightarrow \Pi_1^ {\mathsf H}(\H).$$

\end{thm}

\begin{proof}
Let $\{\gamma_X\}$ be an element of $\Pi_1(\H,1)$, that is a family  of elements of  $\Gamma_X$ where $X$ varies amongst the connected gradings of the $k$-category, and  $\mu(\gamma_X)=\gamma_Y$ for every morphism of connected gradings $\mu:X \to Y$. By definition, $\tau\{\gamma_X\}$ is the subfamily given by the connected gradings which are also Hopf gradings of $\H$. Since the entire family is compatible with respect to morphisms of gradings, the subfamily verifies $\mu(\gamma_X)=\gamma_Y$ for any morphism $\mu$ of connected Hopf gradings. \qed
\end{proof}

\begin{rem}
The fundamental group of a $k$-category is not abelian in general, while the fundamental group of a Hopf category is always abelian. Consequently the morphism $\tau$ is not injective in general. Nevertheless we will see in the next section that $\tau$ is injective for Taft categories. On the contrary, for the Hopf algebra of $k$-valued maps over cyclic groups $\tau$ may not be injective.
\end{rem}

In what follows we will study the morphism $\tau$ whenever universal covers exist. As already mentioned, we recall from \cite{CRS JNG} that a connected grading $U$ of a $k$-category $\B$ is \emph{universal} if for any connected grading $X$ of $\B$ there exists a unique morphism of gradings $\mu: U\to X$. If it exists, the fundamental group $\Pi_1(\B, b_0)$ is isomorphic to $\Gamma_U$.

\begin{defi}
Let $\H$ be a Hopf $k$-category over a finite group. A connected Hopf grading $V$ of $\H$ is \emph{universal} if for any connected Hopf grading $Y$ of $\H$ there exists a unique morphism $\mu:V\to Y$.
\end{defi}

In subsequent sections we will consider examples of Hopf $k$-categories; some of them admit a universal Hopf covering, while others do not.

\begin{lem}
Let $\H$ be a Hopf $k$-category over a finite group admitting a universal Hopf grading $V$. Let $Y$ and $Y'$ be connected Hopf gradings. There exists at most one morphism of connected Hopf gradings $Y\to Y'$.
\end{lem}
\begin{proof}
For a Hopf connected grading $Y$ let $\mu_Y$ be the unique morphism $\Gamma_V\to \Gamma_Y$. Let $\mu_1$ and $\mu_2 : \Gamma_Y\to \Gamma_{Y'}$ be morphisms of Hopf gradings. We infer $\mu_1 \mu_Y = \mu_{Y'} = \mu_2 \mu_Y$. Recall that morphisms of connected Hopf gradings are surjective by Remark \ref{surjectiveHopf}. Hence $\mu_1=\mu_2$.
\qed
\end{proof}

\begin{pro}\label{universal}
Let $\H$ be a Hopf $k$-category over a finite group admitting a universal Hopf grading $V$. There is an isomorphism
$$\varphi : \Pi_1^{\mathsf H}(\H)\longrightarrow\Gamma_V.$$
\end{pro}

\begin{proof}
Let  $\beta$ be an element of $\Pi^{\mathsf H}_1(\H)$, that is a family $\{\beta_Y\}$ where $Y$ varies amongst the connected Hopf gradings of $\H$ and $\mu(\beta_Y)=\beta_{Y'}$ for every morphism $\mu:Y\to Y'$. We define $\varphi(\beta)=\beta_V$.

In order to prove that $\varphi$ is injective, assume that  $\beta_V=0$. Let $Y$ be any Hopf connected grading of $\H$ and let $\mu_Y$ be the unique morphism of connected Hopf gradings $\mu_Y: V\to Y$. Since $\mu(\beta_V)=\beta_Y$ we infer that $\beta_Y=0$, hence $\beta=0$.

Let $c\in\Gamma_V$. Consider the family $\beta$ defined by $\beta_Y=\mu_Y(c)$, where $\mu_Y$ is as before the unique morphism $V\to Y$. This family is an element of $\Pi^{\mathsf H}_1(\H)$ such that $\varphi(\beta)=c.$

\qed
\end{proof}

\begin{thm}
Let $\H$ be a Hopf $k$-category. Suppose that $\H$ admits both a universal grading $U$ as a $k$-category and a universal Hopf grading $V$. The morphism
$$\tau:\Pi_1(\H,1)\longrightarrow \Pi_1^ {\mathsf H}(\H)$$
is surjective.
\end{thm}

\begin{proof}
Let  $\gamma$ be an element of $\Pi^{\mathsf H}_1(\H)$, that is a family $\{\gamma_Y\}$ where $Y$ varies amongst the connected Hopf gradings of $\H$ and $\mu(\gamma_Y)=\gamma_{Y'}$ for every morphism $\mu:Y\to Y'$. Since $V$ is universal,  the previous proposition shows that $\gamma_V$ determines $\gamma$. As well,  any family  $\beta =\{\beta_X\}\in \Pi_1(\H,1)$ is determined by $\beta_U$. Consider the unique morphism of connected gradings $\mu_V: \gamma_U\to \gamma_V$, which is surjective by Remark \ref{surjective}. Let $\beta_U$ be an element in the preimage of $\gamma_V$ by $\mu$. The previous proposition shows that this element determines a unique element $\beta \in \Pi_1(\H,1)$. Observe that the value of the family $\left(\tau(\beta)\right)$ at $V$ is $\gamma_V$, hence $\tau(\beta)=\gamma$.

\qed
\end{proof}

\section{\sf Free Hopf $k$-categories over a group and Taft categories}
\normalsize

In order to define a free $k$-category the following data is required: a set $\mathcal{L}_0$ (which will be the set of objects) and a family of vector spaces
$V=\{ {}_yV_x\}_{x,y \in \L_0}$ (which will be the free generators). To each sequence  $u=(u_n, \cdots, u_0)$  of elements in $\L_0$ we attach a vector space
$$W(u) =\li{u_n}{\!V}_{u_{n-1}}\otimes \cdots \otimes\li{u_2}{\!V}_{v_1}    \otimes  {}_{u{_1}}\!V_{u_{0}}.$$
 For a singleton sequence $u=(u_0)$ we set $W(u_0) =k$.

 The set of finite sequences of elements of $\L_0$ which begin with $x$ and end with $y$ is denoted ${}_y{T}_x$. The concatenation of two  sequences
 $$v=(z=v_m, \dots, v_0=y) \mbox{\ and\ } u=(y=u_n,\dots , u_0 =x)$$ which belong respectively to  ${}_z{T}_y$ and ${}_y{T}_x$ is the sequence
 $$vu =(z=v_m, \dots, v_0=y=u_n, \dots ,u_0=x)\in {}_z{T}_x.$$

Finally let $ \li{y}{W}_x$ be the vector space ${\bigoplus_{u \in {\li{y}{T}_x}}} W(u).$

\begin{defi}\label{definition of free} Let $\L_0$ be a set and let $V$ be a family as above. The \emph{free $k$-category $\L_k(V)$} has set of objects $\L_0$. The vector space of morphisms   from   $x\in\L_0$ to $y\in\L_0$ is ${}_y W_x$. The composition is given by the direct sum of the tensor product maps $$W(v)\otimes W(u) \to W(vu).$$ The vector spaces of the family $V$ are called the \emph{generating vector spaces}.
 \end{defi}
Observe that this definition is intrinsic, in the sense that it does not depend on the choice of  bases of the vector spaces of the family. A choice of bases provides a quiver $Q$ attached to the data as follows: vertices are the elements of $\L_0$, while the set of arrows from $x$ to $y$ is the basis of  $\li{y}{V}_x $.

Let $\C$ be a  small category. Its \emph{linearisation} $k\C$ is the $k$-category with the same set of objects. Its morphisms are the vector spaces with bases the morphisms of the category. The composition of $k\C$ is obtained by extending bilinearly the composition of $\C$.

Let $Q$ be a quiver. There is a free category $\F_Q$ (see for instance \cite[p. 48]{maclane}) which objects are the vertices of $Q$ and which morphisms are the sequences of concatenated arrows (called paths), including the trivial one at each vertex.

The composition is the concatenation of paths, and the trivial paths are the identities.  The following result is clear.

 \begin{pro}
 Let $\L_k(V)$ be a free $k$-category associated to the data $\L_0$ and $V$, and let $Q$ be a quiver attached to the data. The linearisation $k\F_Q$ and the free $k$-category $\L_k(V)$ are isomorphic.
\end{pro}

  We consider now a finite group $G$ which elements will become the objects of a free $k$-category. In order to obtain a family of vector spaces as above,  let $B$ be a $k^G$-Hopf bimodule. Since $k^G$ is a semisimple algebra, $B$ has a unique decomposition as a direct sum of its isotypic $k^G$-bimodule components $\delta_yB\delta_x$, where $\delta_x$ denotes the Dirac mass on $x$. Let $\H_B$ be the free $k$-linear category with set of objects $G$ and whose family of vector spaces is $\{ \delta_yB\delta_x\}_{x,y\in G}$.

  \begin{pro}(compare with \cite{CR,CR2})
  The free $k$-category $\H_B$ has a structure of Hopf $k$-category over $G$.
  \end{pro}
  \begin{proof} Through the structure of $B$, we obtain a comultiplication on each generating vector space as the sum of the left and right comodule structure morphisms, and this is linearly extended in the unique possible way to a comultiplication on the whole free $k$-category. There exist a counit and an antipode providing $\H_B$ with a structure of Hopf $k$-category over $G$, see \cite{CR2}.
  \end{proof}
\normalsize

\begin{pro}\cite{CR,CR2}
Any Hopf $k$-category over a finite group $G$ whose underlying category is free is isomorphic to a Hopf $k$-category given by the preceding construction.
\end{pro}

\begin{exa}\cite{C}\label{taft category}
Let $C_n=\langle t|t^n=1\rangle$ be a cyclic group of order $n\ge 2$. The $n$-th crown $k$-category $\C_n$ is the free $k$-category with set of objects $C_n$ and generating morphisms  $\{a_u:u\to tu\}_{u\in C_n}$. Given $q\in k^{*}$, the category $\C_n$ has a structure of a  Hopf $k$-category, with comultiplication induced by
\[ {}_{t^{i+1}}\Delta_{t^i}(a_{t^i})= \sum_{t^jt^k=t^i}t^j\otimes a_{t^k}+ q^ka_{t^j}\otimes t^k. \]
See Lemma 3.1 of \cite{C}.

 Recall that a \emph{two-sided ideal} of a $k$-category is a family $I$ of subvector spaces of the morphisms such that the quotients inherit a well defined composition, which provides a $k$-category denoted $\C/I$.

 Let $I$ be the smallest two-sided ideal of $\C_n$ containing all the compositions of length $n$ of the $a_i$'s. If $q$ is a primitive $n$-th root of unity, then $I$ is a \emph{Hopf ideal}, meaning that $\C_n/I$ inherits a structure of Hopf $k$-category over the same group. The proof of this fact can be obtained from \cite{C}.
\end{exa}

\begin{defi}
The $n$-\emph{Taft category}  $\T_{n,q}$ is $\C_n/I$.
\end{defi}

\begin{lem}
Let $X$ be a connected grading of the underlying $k$-category $\T_{n,q}$. Then $\Gamma_X$ is cyclic.
\end{lem}

\begin{proof}
 Note that $X$ is not necessarily a Hopf grading. Since the spaces of morphisms are at most one dimensional, they are homogeneous.
Let $w$ be the homogeneous walk $w=\left(a_{t^{n-1}},1\right),\dots, \left(a_t,1\right),\left(a_1,1\right)$. The image of the homomorphism $\mathsf{deg}_X:\li{1}{HW}(\C_{n,q})_1 \to \Gamma_X $   is cyclic, generated by $\mathsf{\deg}_Xw=(\mathsf{\deg}_Xa_{t^{n-1}})\dots(\mathsf{\deg}_Xa_{t})(\mathsf{\deg}_Xa_{1})$.
\end{proof}

\begin{lem}
Let $X$ be a connected Hopf grading of the $n$-Taft $k$-category $\T_{n,q}$.
The generating morphisms share the same degree and the grading group $\Gamma_X$ is cyclic of order coprime with $n$. There is no universal grading.
\end{lem}
\noindent\textbf{Proof. }\sf
The comultiplication is homogeneous with respect to $X$. The  formula of Example \ref{taft category} for the comultiplication in a Taft category implies that all the generators of the category have the same degree that we call $\gamma$. Using $w$ of the preceding proof, we obtain that $\deg_X w= \gamma^n$. Note that $\gamma^n$ is a generator
of $\Gamma_X$, which is thus of some finite order $m$ coprime with $n$. The last assertion follows from this. \qed

We compute now the fundamental group of the Taft $k$-categories. Since only cyclic groups $C_m$ of order $m$ coprime with $n$ can occur as groups of connected Hopf gradings of  this category, the next result follows immediately.
\begin{pro}\label{inverse limit}
The fundamental group of the $n$-th Taft category $\T_{n,q}$ is the subgroup of $\prod_{(n,m)=1}\mathds{Z}/m\mathds{Z}$ consisting of families $(x_m)$ verifying $\mu(x_m)=x_{m'}$, where
$\mu:\mathds{Z}/m\mathds{Z}\to \mathds{Z}/m'\mathds{Z} $
is the canonical projection in case $m'$ divides $m$.

 $$\Pi_1^{\mathsf H}(\T_{n,q}) = \varprojlim_{(m,n)=1}\mathds{Z}/m\mathds{Z}.$$

\end{pro}

Next we  consider the morphism $\tau$ of Theorem \ref{themorphism} for an $n$-th Taft category. In this situation the space of homomorphisms between two objects are either zero or one dimensional, those categories are called \emph{Schurian}. In \cite{CRS DOC 11} a $CW$-complex is associated to such categories, its fundamental group from algebraic topology is isomorphic to the fundamental group of the $k$-category. In case of an $n$-th Taft category it is straightforward to infer that $\Pi_1(\T_{n,q},1)$ is isomorphic to $\mathds{Z}$. Moreover, since the $k$-category is Schurian,  there is a universal covering, that is a universal grading by $\mathds{Z}$. Note that the structure of $\T_{n,q}$ as a  $k$-category does not depend on $q$, this element of $k$ is only involved in the Hopf structure of the $k$-category. Consequently the fundamental group of the underlying $k$-category does not depend on $n$.

\begin{pro}
There is a commutative diagram

\begin{displaymath}
    \xymatrix{
        \Pi_1(\T_{n,q},1) \ar[r]^\tau \ar[d] & \Pi^{\mathsf H}_1(\T_{n,q}) \ar[d] \\
        \mathds{Z}  \ar[r]  & \varprojlim_{(m,n)=1}\mathds{Z}/m\mathds{Z}}
\end{displaymath}
where the left vertical arrow is the isomorphism obtained from \cite{CRS DOC 11}, the right vertical arrow is the isomorphism of Proposition \ref{inverse limit}, and the bottom arrow is induced by the canonical projections  to the quotients $\mathds{Z}/m\mathds{Z}$ for $(m,n)=1$. This last map is injective.

\end{pro}

\begin{proof}
There are infinitely many primes  not dividing $n$. Hence if $x$ is in the kernel there are infinitely many primes dividing $x$, so $x=0$.
\qed
\end{proof}

\section{\sf Hopf algebra of $k$-valued maps over a cyclic finite group}

In this section we consider the Hopf algebra $k^{C_n}$ as a Hopf $k$-category with a single object in order to compute its fundamental group.

The next table summarises our results for an algebraic closed field $k$ of characteristic different from $2$ and $3$. In fact, the computations bellow provide results for any field, but  under the previous assumption they can be presented in a uniform way.
\normalsize

\begin{center}
\renewcommand{\arraystretch}{1.8}
\begin{tabular}{|c|c|c|c|c|c|c|}
  \hline
  n & 2 & 3 & 4 & 5 & 6 & 7 \\
  \hline
  $\Pi_1^{\mathsf H}(k^{C_n})$ & $0$ & $C_2$ & $C_2$ & $C_4$ & $C_2$ & $C_6$ \\
  \hline
\end{tabular}
\end{center}

\medskip

Let $k^E$ be the $k$-algebra of maps over a finite set $E$. Its connected gradings have been classified by S. D{\u{a}}sc{\u{a}}lescu, see \cite{dasca}, see also \cite{ginsch}. However the maps between those gradings are intricate, see  \cite{CRS ANT 10}. At the opposite, we will show that for $n$ at most $7$, the list of connected Hopf gradings  is short and  the morphisms between them can be described.

Recall that a central homogeneous idempotent enables to decompose a grading as follows.

\begin{lem}
Let $B$ and $C$ be $k$-algebras with gradings $X$ and $Y$ by the same group $\Gamma$. The algebra $B\times C$ is also graded by $\Gamma$.
\end{lem}
\begin{proof}
Consider the grading $Z$ given for $s\in\Gamma$, by $Z^s(B\times C)=X^sB\times Y^sC$.\qed
\end{proof}

\begin{lem}\label{gr-decomposition}
Let $A$ be a $k$-algebra with a grading $X$. Let $e$ be a central homogeneous idempotent of $A$. The grading $X$ induces a grading on the algebra $Ae$ with unit $e$. Moreover $A=Ae\times A(1-e)$ as graded algebras.
\end{lem}
\begin{proof}
Note first that $\deg_X(e) =1$. We consider a decomposition of $Ae$ as a direct sum of the vector spaces  $(X_e)^s(Ae):=(X^sA)e$ where $s\in\Gamma_X$. Observe that $X^sA=  (X^sA)e\oplus (X^sA)(1-e)$ and $(X^sA)e=Ae\cap X^sA$. Moreover it is easily verified that $X_e$ is a grading of the algebra $Ae$ and the conclusion follows.\qed
\end{proof}

We recall or establish some facts concerning gradings of commutative semisimple algebras.

Consider on one hand the category of $k$-algebras of the form $k\times\cdots\times k$ and on the other hand the category of finite sets. These categories are isomorphic through  canonical contravariant functors: to a set $E$ we associate the $k$-algebra $k^E$, while to a $k$-algebra $k\times\cdots\times k$ we associate its complete set of primitive central idempotents $\left\{(1,0,\dots,0),\dots , (0,\dots,0,1)\right\}$.

\begin{pro}
Let $E$ be a finite set and let $A=k^E$. The subalgebras of $A$ are in one-to-one correspondence with the partitions of $E$.  The elements of the  subalgebra corresponding to a partition $P$ are the $k$-valued maps over $E$ which are constant on the subsets of the partition. This subalgebra is isomorphic to $k^P$.
\end{pro}
\begin{proof}
Subalgebras correspond to surjective maps between finite sets in the above contravariant isomorphism of categories. Observe that a surjective map is just a partition of the source set provided by the fibers of the map.
\qed
\end{proof}

\begin{lem}
Let $E$ be a finite set and let $A=k^E$. The idempotents of $A$ are in one-to-one correspondence with the subsets of $E$.
\end{lem}
\begin{proof}
 To a subset $F$ of $E$ we associate the $k$-valued idempotent map $\delta_F$ which is the characteristic map of $F$. Conversely to an idempotent map $e$ we associate its support. \qed
\end{proof}

\begin{defi} (see \cite{dasca})
Let $A$ be an algebra. A grading $X$ of $A$ is \emph{ergodic} if $X^1A$ is one dimensional.
\end{defi}

\begin{thm}
Let $E$ be a finite set and let $X$ be a grading of $k^E$. There exists a unique decomposition of $k^E$ as a product of algebras with ergodic gradings.

Moreover, there exists a unique partition $P_X$ of $E$ such that the ergodic graded algebras involved in the decomposition are the algebras $k^F=k^E\delta_F$, where $F$ is an element of $P_X$.
\end{thm}

\begin{proof}
According to the previous results, the subalgebra of homogeneous elements of trivial $X$-degree of $k^E$ determines a partition $P_X$ of $E$; this subalgebra is $k^{P_X}$, that is the $k$-valued maps on $E$ which are constant on the sets of $P_X$. In other words
$$X^1(k^E) = \bigtimes_{F\in P_X} k\delta_F.$$

This subalgebra has a complete system of central primitive idempotents
  $$\{\delta_F\}_{F\in P_X}.$$
In turn and according to Lemma \ref{gr-decomposition}, each algebra $k^E\delta_F$ is graded, and there is an isomorphism of graded algebras:
$$k^E = \bigtimes_{F\in P_X} k^E\delta_F.$$
The elements of the algebra $k^E\delta_F$ are the maps which support is contained in $F$, that is $k^F$, with unit element $\delta_F$. Its  subalgebra of homogeneous elements of degree $1$ is one dimensional: $$X_{\delta_F}^1(k^E\delta_F) = (X^1k^E)\delta_F= k\delta_F.$$
\qed
\end{proof}

\begin{thm}\label{superergodic} (see also \cite{dasca})
Let $F$ be a finite set with a fixed element $u_0$, and let $X$ be a grading of $k^F$ by a group $\Gamma_X$. If $X$ is ergodic, then in each non zero homogeneous component $X^s(k^F)$ there exists a unique homogeneous \emph{normalised} map $a_s$ of degree $s$, that is verifying $a_s(u_0)=1$. Moreover $a_s$ is invertible, of finite order, and the homogeneous component is one dimensional.

If  $X$ is connected, then the family $\{a_s\}_{s\in \Gamma_X}$ is a multiplicative basis of $k^F$. This provides an isomorphism of $k^F$ with the group algebra $k\Gamma_X$.
\end{thm}

Before proving this result, we recall the following.

\begin{lem}
Let $A$ be a finite dimensional $k$-algebra with a grading $X$. Given $a\in A$ a homogeneous and non nilpotent element, there exists a positive integer $n$ such that $a^n$ is of trivial degree.
\end{lem}
\begin{proof}
  Note that for all $i$ we have $a^i \neq 0$ and $\deg_Xa^i= \left(\deg_Xa\right)^i$. The set $\{a^i\}_{i\geq0}$ is contained  in the finite dimensional vector space $A$, hence there exists a non trivial dependence relation between its elements. We consider a dependence relation of minimal length - which is at least $2$. All the powers of $a$ affording non zero coefficients in this relation have the same $X$-degree, in particular there exist $i\neq j$ such that $ \left(\deg_Xa\right)^i = \left(\deg_Xa\right)^j$, so there exists $n$ such that $\left(\deg_Xa\right)^n=1$.
\qed
\end{proof}

\noindent\textbf{Proof ot Theorem \ref{superergodic}.} There are no nilpotent elements in $k^F$. According to the Lemma above, for every non zero homogeneous element $a$ there exists $n$ such that $a^n\in X^1(k^F)=k$. This implies that $a$ is invertible, meaning that  $a(x)\neq 0$ for every $x\in F$. We normalise $a$ by considering $a'=a(u_0)^{-1}a$. Since $a'$ is homogeneous of the same degree than $a$, we infer that $(a')^n\in X^1(k^F)$. This means that all the coefficients of $(a')^n$ are equal, then equal to $1$ since we already know that $(a')^n(u_0)=1$.

Let $a$ and $b$ be homogeneous normalised elements of  the same degree, their difference is homogeneous of the same degree. However $a-b$ is not invertible since $(a-b)(u_0)=0$, hence $a-b=0$.

Finally we consider the support of the grading $X$, that is the set of $s \in \Gamma_X$ such that $X^s(k^F)\neq 0$. We assert that it coincides with $\Gamma_X$. Since $X$ is connected, its support generates $\Gamma_X$ (see Lemma \ref{support}). The previous considerations prove that the support is in fact already a subgroup of  $\Gamma_X$, then equal to $\Gamma_X$.
\qed

We summarise the results above as follows. Let $X$ be a grading of $k^E$, where $E$ is a finite set. There exists a canonical partition $P_X$ of $E$ which is determined by the subalgebra of homogeneous elements of trivial degree. A complete set of idempotents is inferred, namely the characteristic maps of the sets of the partition. In turn, this provides a decomposition of $k^E$ as a product of graded ergodic algebras  $k^F$, where $F$ varies along the subsets of $E$ given by the partition $P_X$, and $k^F$ is identified to the algebra of $k$-valued maps on $E$ which support is contained in $F$. The dimension of $k^F$ is the cardinality of $F$. Since the grading of $k^F$ is ergodic, all its non zero homogeneous components are one dimensional.

\medskip

In particular a grading $X$ is trivial (that is $k^E = X^1(k^E)$) if and only if its associated  partition is the finest one:  $P_X=\{\{x\ \}\}_{x\in E}$.

\medskip

Let $C_n= \langle t\mid t^n=1\rangle $ be the cyclic group of order $n$ with a given generator $t$. The $k$-algebra $k^{C_n}$ is a Hopf $k$-algebra, see Remark \ref{dualofagroup}.

\begin{lem}
Let $X$ be a Hopf grading of $k^{C_n}$. The corresponding partition $P_X$ of $C_n$ is preserved under the antipode $S$, \emph{i.e.} if $F\in P_X$ then $S(F)\in P_X$.
\end{lem}
\begin{proof}
The antipode of $k^{C_n}$ is homogeneous, in particular it preserves the subalgebra $X^1( k^{C_n})$. The set of characteristic maps of the subsets of the partition is the complete set of orthogonal primitive idempotents of this algebra. Note that this set is unique. Consequently $S$ preserves this set, hence it preserves the partition.
\end{proof}

\begin{lem}
Let $X$ be a Hopf grading of $k^{G}$, where $G$ is a finite group.  Any homogeneous element $a$ of non trivial degree verifies $a(1)=0$.
\end{lem}
\begin{proof}
The counit $\epsilon : k^G\to k$ is homogeneous, in particular $\epsilon(a)=0$. Note that $\epsilon(a)=a(1).$\qed
\end{proof}
\begin{pro}\label{one}
Let $X$ be a Hopf grading of $k^{C_n}$ and let $P_X$ be the corresponding partition of $C_n$. The singleton $\{1\}$, where $1$ is the trivial element of $C_n$, is one of  the sets of the partition.
\end{pro}
\begin{proof}
Let $F_1$ be the set of the partition containing $1$. If $|F_1| \geq 2$, then the dimension of $k^{F_1}$ is at least $2$. Since $k^{F_1}$ is ergodic, there exists a homogeneous component of non trivial degree. Let $h$ be a non zero element in this component. Then $h$ is invertible by Proposition \ref{superergodic} and in particular $h(1)\neq 0$, which contradicts the previous result. \qed
\end{proof}

We are now in position to describe the connected Hopf gradings of $k^{C_n}$, for $n\leq 7.$

\begin{pro}
For any field $k$, the fundamental group of the Hopf algebra $k^{C_2}$ is trivial.
\end{pro}
\begin{proof}
Let $X$ be a Hopf grading of $k^{C_2}$. According to the previous proposition, the corresponding partition includes $\{1\}$, hence $P_X$ is the finest partition of $C_2$. We have  already noticed that this partition corresponds to the trivial grading.
\end{proof}

\begin{pro}\label{C3}
For any field $k$ of characteristic different from $2$, there is a unique and hence universal non trivial connected Hopf grading of $k^{C_3}$ by the group $C_2$. Consequently $\Pi_1^{\mathsf H}(k^{C_3}) = C_2.$
In characteristic $2$ this fundamental group is trivial.
\end{pro}
\begin{proof}
Our purpose is to describe the non trivial connected Hopf gradings of $k^{C_3}$. Let $X$ be such a Hopf grading.  We know that $P_X$ is not the finest partition and that $\{1\}$ is one of its subsets. Consequently  $P_X= \{ \{1\}, F\}$ where $F=\{t, t^2\}$. Observe that this partition is preserved by $S$, hence $k^{C_3}$ is the product $k\times k^F$ of ergodic graded algebras.

Since $k^F$ is ergodic, there exists a homogeneous element $h$ of non trivial degree $s$, which is then invertible. We normalise at $t$, obtaining  $h=\delta_t + x\delta_{t^2}$ with $x\neq 0$. Observe that the homogeneous component of degree $s$ is one dimensional in $k^F$ and in $k^E$, hence it has to be preserved by $S$. Since $S(h)=x\delta_t + \delta_{t^2}$, we infer $x^2=1$.

Of course $\delta_F$ and $h$ are linearly independent, so $x=1$ is not possible. If the characteristic of $k$ is $2$, then there are no Hopf gradings of $k^{C_3}$.

If the characteristic of $k$ is not $2$, consider $h=\delta_t - \delta_{t^2}$ which is homogeneous of degree $s$, and note that $s^2=1$.

 The comultiplication of $k^{C_3}$ has to be homogeneous with respect to $X$. Equivalently  we consider $Y$, the dual direct sum decomposition of $kC_3$ deduced from $X$, in order to see if it is a grading of $kC_3$ or not. For this purpose we search for a basis adapted to the dual decomposition $Y$. As a pattern for future reference, we next  provide the details of the method we use for obtaining such a basis.

 First we consider the matrix $M$ of the homogeneous basis vectors $\{\delta_1, \delta_F, h\}$ expressed in terms of the Dirac masses:
$$M= \left(
       \begin{array}{rrr}
         1 & 0 & 0 \\
         0 & 1 & 1 \\
         0 & 1 & -1 \\
       \end{array}
     \right)
     .$$
The corresponding basis for the dual decomposition, expressed in the dual basis, namely in the elements of the group, is given by the matrix
$$\left(M^{-1}\right)^t=
\left(
       \begin{array}{rrr}
       1 &0 &0\\
       0& 1/2 &1/2 \\
       0 &1/2 & -1/2\\
  \end{array}
     \right).
$$
Observe that  in this case the vectors $\{\delta_1, \delta_F, h\}$  are  orthogonal with respect to the canonical  inner product in the Dirac masses. Then, the inverse of each block of $M$ equals its transpose up to a non zero element of the field.

We obtain a basis of $kC_3$ adapted to the direct sum decomposition $Y$ where $\{1, t+t^2\}\subset Y^1(kC_3)$ and $\{t-t^2\}\subset Y^s(kC_3)$.  It is immediate to check that $Y^1(kC_3)$ is a subalgebra. Moreover
$$Y^1(kC_3) Y^s(kC_3) \subset Y^s(kC_3) \mbox{ and } \left(Y^s(kC_3)\right)^2 \subset Y^1(kC_3).$$

Hence there exists a unique non trivial connected grading of $k^{C_3}$ which is universal. Proposition \ref{universal} provides the result.\qed
\end{proof}

\begin{pro}
  If the characteristic of $k$ is different from $2$, there is a unique non trivial connected Hopf grading of $k^{C_4}$ and $\Pi_1^{\mathsf H}(k^{C_4}) = C_2$ . In characteristic $2$ this fundamental group is trivial.
\end{pro}
\begin{proof}
We first provide the list of  partitions of $C_4$ which are preserved by the antipode $S$, excluding the finest one. According to Lemma  \ref{one}, the subset $\{1\}$ has to be a subset of each partition.
    \begin{itemize}
    \item $A=\{ \{1\}, \{t^2\}, \{t, t^3\} \},$
    \item $B= \{ \{1\},  \{t, t^2, t^3\} \}.$
    \end{itemize}

For $A$, the graded decomposition of a possible Hopf grading $X$ of $k^{C_4}$ is $$k\delta_1 \times k\delta_{t^2} \times k\delta_F$$ where $F=\{t, t^3\}$. Concerning the ergodic grading of $k\delta_F$, the situation is similar to the previous one, it should exist  an invertible element $h=\delta_{t}+x\delta_{t^3} \in k^F$, homogeneous of non trivial degree $s$ with $s^2=1$. Since $S$ has to preserve $kh$, we obtain $x^2=1$. In characteristic $2$, we have $h=\delta_F$ which is impossible and there is no grading corresponding to this partition. Otherwise, with $x=-1$ we notice that the basis vectors $\{\delta_1, \delta_{t^2}, \delta_F, h\}$ are orthogonal with respect to the inner product in the Dirac masses. Let $Y$ be the dual decomposition of $kC_4$. According to the method of the preceding proof, we obtain $$Y^1 (kC_4) = k\oplus kt^2 \oplus k(t+t^3) \mbox{ and } Y^s(kC_4)=k(t-t^3).$$ One can verify that this decomposition provides a grading of the algebra $kC_4$.

\medskip
 For $B$,  the graded decomposition of a possible Hopf grading  of $k^{C_4}$ is $$k\delta_1  \times k\delta_F$$ where $F=\{t,t^2, t^3\}$. The grading induced on $k^F$ is ergodic, let $h=\delta_t+x\delta_{t^2}+y\delta_{t^3}$ be homogeneous of degree $s$.  The element $h$ is invertible in $k^F$ and $S(h)$ is a scalar multiple of $h$. Now $S(h)=y\delta_t+x\delta_{t^2}+\delta{t^3}$, so $S(h)=yh$. Moreover $y^2=1$ and $yx=x$. We infer $y=1$, then $h^2=\delta_t+ x^2\delta{t^2}+\delta_{t^3}$. Hence  $\{\delta_F, h, h^2\}$ is linearly dependent and there is no Hopf grading corresponding to the partition $B$. \qed

\end{proof}

\begin{pro}
If $k$ contains a  root of unity of order $4$, then  there is a universal Hopf grading with group $C_4$ and $\Pi_1^{\mathsf H}(k^{C_5}) = C_4$.

If the characteristic is not $2$ and there is no root of unity of order $4$, there exists a unique non trivial  Hopf grading and this fundamental group is $C_2$.

In characteristic  $2$, this fundamental group is trivial.
\end{pro}

\begin{proof}
We proceed as in the previous proposition,  the list of possible partitions of $C_5$ for a Hopf grading of $k^{C_5}$ is as follows.
 \begin{itemize}
    \item $A=\{ \{1\}, \{t,t^4\}, \{t^2, t^3\} \}$ and we set $F=\{t,t^4\}$, $G=\{t^2, t^3\}$,
    \item $B= \{ \{1\},  \{t, t^2, t^3, t^4\}\}$ and we set $ H=\{t, t^2, t^3, t^4\}$,
    \item $C= \{ \{1\}, \{t\}, \{t^2, t^3\}, \{t^4\} \}$ and $C'$ obtained from $C$ with $t^2$ as a generator instead of $t$,
    \item $D= \{ \{1\}, \{t,t^2\}, \{t^3, t^4\} \}$ and $D'$ as above.
    \end{itemize}

We will prove the following:
\begin{enumerate}
\item If $k$ contains a root of unity of order $4$, then there are precisely two connected Hopf gradings: \\
- for the partition $A$ with group $C_2$,\\
- for the partition $B$ with group $C_4$.\\
We will observe that the first one is a quotient of the second one.

\item If $k$ does not contain a root of unity of order $4$ and is not of characteristic $2$, then there is one connected Hopf grading with corresponding partition $A$ and group $C_2$.

\item If $k$ is of characteristic $2$, there are no non trivial connected Hopf gradings.
\end{enumerate}
Let $X$ be a possible Hopf grading with partition $A$, the decomposition as graded algebra of $k^{C_5}$  is $$k\delta_1\times k\delta_F\times \delta_G.$$  The graded algebras $k^F$ and $k^G$ are ergodic, and each one is preserved by $S$. As before, if the characteristic of $k$ is $2$ then there is  no Hopf grading.

Otherwise let $f=\delta_t-\delta_{t^4}$ and $g=\delta_{t^2}-\delta_{t^3}$ be homogeneous elements of respective degrees $s$ and $s'$, both of them are of order $2$. Observe that all the vectors involved are orthogonal to each other with respect to the inner product in the Dirac masses. Hence we obtain a direct sum decomposition $Y$ of  $kC_5$:

$Y^1(kC_5)= k\oplus k(t+t^4) \oplus k(t^2+t^3)$

$Y^s(kC_5)=k(t-t^4)$

$Y^{s'}(kC_5)= k(t^2-t^3).$

Simple computations show that $Y^1(kC_5)$ is a subalgebra of $kC_5$. Moreover $Y^1Y^s=Y^s$ and $Y^1Y^{s'}=Y^{s'}$.  Finally, since $Y^{s}Y^{s'}\subset Y^1$ it follows that $ss'=1$, hence $s=s'$. This way we have confirmed the existence of a unique Hopf grading by $C_2$ of $k^{C_5}$ with partition $A$.

\medskip
For  $B$, we have $k^{C_5}=k\delta_1\times k\delta_H$. The graded $k$-algebra $k^H$ is ergodic of dimension $4$. Let $h$ be a homogeneous element of non trivial degree $s$, hence $h$ is invertible in $k^H$. We normalise $h$ at $t$, obtaining $h=\delta_t+ x\delta_{t^2}+ y\delta_{t^3}+ z\delta_{t^4}$. According to Lemma \ref{superergodic},  $h^4=\delta_F$. The homogeneous component containing  $h$ is one dimensional and the antipode preserves it. Since $S(h)(t)=z$, necessarily $S(h)=zh$. We infer $y=zx$, $x=zy$ and $z^2=1$.

 We first consider the case $z=1$, which implies $x=y$. If the order of $s$ is $4$, then the family $\{\delta_H, h, h^2, h^3\}$ is linearly dependent and there is no Hopf grading. The order of $s$ should be $2$, and $h= \delta_t - \delta_{t^2} - \delta_{t^3}+ \delta_{t^4}$. Consider the quotient of the group of this possible grading by the subgroup generated by $s$. In this new grading the element $h$ is of trivial degree, and the associated partition is no longer $B$ but $D$. We will prove below that there are no Hopf gradings with associated partition $D$.

We assume now that the characteristic of $k$ is not $2$, and $z=-1$. Then
 $h=\delta_t+ x\delta_{t^2} -x\delta_{t^3} - \delta_{t^4}$.  We assert that for $x=\pm 1$ there is  no associated Hopf grading. Indeed, in that case $h$ is homogeneous of degree $s$ of order $2$. As above  we consider a quotient of the group of the possible grading by the subgroup generated by $s$. In this new grading $h$ is of trivial degree and the partition $B$ becomes $D$. However there are no Hopf gradings with associated partition $D$, see below.

In the previous paragraph we covered the case $x=\pm 1$, enabling us to conclude that if there is no primitive root of unity of order $4$ in $k$ there are no connected Hopf gradings with partition $B$. Otherwise, let $x=i$ and note that $h$ and $h^3$ are not orthogonal. According to the method that we have considered in the proof of Proposition  \ref{C3}, we compute the inverse of the matrix of the vectors $\delta_F, h, h^2, h^3$ expressed in the basis of Dirac masses, which is as follows:

$$
1/4\left(
       \begin{array}{rrrrr}
       4 &0 &0 &0&0\\
       0& 1 &-i &i& -1\\
       0 &1 & -1&-1&1\\
       0&1&i&-i&-1\\
  \end{array}
     \right).
$$

The rows of this matrix provide a basis for the dual decomposition $Y$ of $kC_5$, for instance $Y^1= k\oplus k(t+t^2+t^3+t^4)$ and $Y^s= k(t-it^2+it^3-t^4)$. Several non difficult computations show that $Y$ is a grading of $kC_5$. This way we obtain a connected Hopf grading of $k^{C_5}$ with partition $B$ and group $C_4$.

Observe that  choosing $x=-i$, the Hopf grading obtained is isomorphic to the preceding one. Indeed, the isomorphism is given by the automorphism of $C_4$ which sends $s$ to $s^3$, while the automorphism $J$ is the identity (see Definition \ref{morphism Hopf gradings}).

\medskip
For  $C$,  the decomposition as graded algebras of $k^{C_5}$ is $$k\times k\delta_t \times \delta_{\{t^2, t^3\}} \times \delta_{t^4}.$$ If the characteristic is not $2$, let $h=\delta_{t^2}-\delta_{t^3}$ be the homogeneous element of non trivial degree of  $k^{\{t^2,t^3\}}$. The basis vectors are mutually orthogonal with respect to the inner product in the Dirac masses. We infer a basis for the dual decomposition $Y$ of $kC_5$, and it is easily checked that $Y^1$ is not a subalgebra. We have proved that there are no connected Hopf gradings with partition $C$.

\medskip
Finally for $D$, the homogeneous elements of non trivial   degrees of order $2$ are $h=\delta_t-\delta_{t^2}$ and $S(h)=\delta_{t^4}-\delta_{t^3}$. The involved vectors are mutually orthogonal. As before, the dual decomposition $Y$ of $kC_5$ is not a grading since $Y^1$ is not a subalgebra.

\medskip

\qed
\end{proof}

\begin{pro}
Let $k$ be a field containing a primitive root of unity of order $6$.  There exists a unique non trivial connected Hopf grading of $k^{C_6}$ by the group $C_2$ and $\Pi_1^{\mathsf H}(k^{C_6}) = C_2$ . In other cases this fundamental group is trivial.
\end{pro}

\begin{proof}
The partitions of $C_6$ which are preserved by $S$ are the following:
 \begin{itemize}
    \item $A=\{ \{1\}, \{t,t^5\}, \{t^2, t^4\}, \{{t^3}\} \} $ and we set $F= \{t,t^5\}$,  $G=  \{t^2, t^4\},$
    \item $A' =\{ \{1\}, \{t,t^5, t^3\}, \{t^2, t^4\}\}$,
    \item $A''= \{ \{1\}, \{t,t^5\}, \{t^2, t^4, t^3\}\}$,
    \item $B= \{ \{1\},  \{t, t^2, t^3, t^4, t^5\} \},$

    \item $C= \{ \{1\},  \{t, t^2, t^3, t^4\}, \{t^5\} \}$ and we set $H=\{t, t^2, t^3, t^4\}$,

    \item $D= \{ \{1\}, \{t, t^2\}, \{t^4, t^5\}, \{t^3\} \}.$
    \end{itemize}
Let $A$ be the partition of a connected Hopf grading. Then $k^F$ has an ergodic grading and as shown previously the characteristic of $k$ is not $2$. As in similar  situations let $f= \delta_t - \delta_{t^5}$ be the homogeneous element of degree $s$, with $s$ of order $2$. Similarly, let $g= \delta_{t^3} - \delta_{t^4}$ be homogeneous of degree $s'$ of order $2$ for $k^G$. The basis vectors of this grading of $k^{C_6}$ are mutually orthogonal with respect to the inner product defined in the Dirac masses. Hence we obtain basis vectors for the dual decomposition $Y$ of $kC_6$, for instance $Y^s=k(t-t^5)$ and $Y^{s'}=k(t^2-t^4)$, while $Y^1=k\oplus k(t+t^5)\oplus k(t^2+t^4)\oplus kt^3$. Simple computations show that if $s=s'$ this decomposition $Y$ is indeed a grading of the algebra $kC_6$; there is only one Hopf connected grading of $k^{C_6}$ with partition $A$, it has group  $C_2$.
\medskip

The methods used previously can be applied in order to show that in all other cases there are no connected Hopf gradings. The case of the partition $C$ is particular, we briefly sketch where it fails. Much as in the previous proof, if there exists a root of unity $i$ of order $4$ in $k$,  there is a homogeneous element of degree $s$ of order $4$. The matrix is the same as before. Most of the requirements in order that the dual decomposition is a grading are satisfied, however it finally fails being a grading.
\qed
\end{proof}

We state the result for $C_7$ without sketching the proof. The computations are longer than before but they follow the same lines.

\begin{pro}
$\Pi_1^{\mathsf H}(k^{C_7}) = C_6$ if the characteristic of $k$ is not $2$. There is no universal Hopf grading.
If the characteristic is $2$, then the fundamental group is trivial.
\end{pro}

We recall that $k^{C_n}$ and $kC_n$ are isomorphic Hopf algebras if $k$ contains all $n$-th roots of unity, in which case of course the characteristic of $k$ does not divide $n$. Consequently, the situation for $kC_n$ is summarised by the table at the beginning of this section.

For small values of the cardinal of a set $E$, the fundamental group of the associative algebra $k^E$ has been computed in \cite{CRS ANT 10}, see also \cite{ginsch}.  Note that $k^E$ has a universal grading only for $| E |=2$. Nevertheless the morphism $\tau$ of Theorem \ref{themorphism} is surjective also for $| E |=3, 4$.

\footnotesize
\noindent C.C.:
\\Institut Montpelli\'{e}rain Alexander Grothendieck (IMAG), UMR 5149\\
Universit\'{e}  de Montpellier, F-34095 Montpellier cedex 5,
France.\\
{\tt Claude.Cibils@umontpellier.fr}

\medskip

\noindent A.S.:
\\IMAS-CONICET y Departamento de Matem\'atica,
 Facultad de Ciencias Exactas y Naturales,\\
 Universidad de Buenos Aires,
\\Ciudad Universitaria, Pabell\'on 1\\
1428, Buenos Aires, Argentina. \\{\tt asolotar@dm.uba.ar}

\end{document}